\documentclass[a4paper,12pt]{amsart}

\usepackage[english]{babel}
\usepackage{amsthm, latexsym, gastex, amssymb,url,listings}
\usepackage[mathcal]{eucal}
\usepackage{graphicx,url}

\copyrightinfo{2011}{Simone Ugolini}

\renewcommand{\phi}[0]{\varphi}
\renewcommand{\theta}[0]{\vartheta}
\renewcommand{\epsilon}[0]{\varepsilon}

\newcommand{\N}{\text{$\mathbf{N}$}}
\newcommand{\Z}{\text{$\mathbf{Z}$}}
\newcommand{\Q}{\text{$\mathbf{Q}$}}

\newcommand{\Pro}{\text{$\mathbf{P}^1$}}
\newcommand{\F}{\text{$\mathbf{F}$}}

\DeclareMathOperator{\E5}{End_{\F_5}}

\newtheorem{theorem}{Theorem}[section]
\newtheorem{lemma}[theorem]{Lemma}
\newtheorem{corollary}[theorem]{Corollary}

\theoremstyle{definition}

\newtheorem{example}[theorem]{Example}

\theoremstyle{remark}

\numberwithin{equation}{section}

\begin{document}

\bibliographystyle{amsalpha}

\date{}

\title[GS]
{Graphs associated with the map $X \mapsto X + X^{-1}$
in finite fields of characteristic  five}

\author{S.~Ugolini}
\email{sugolini@gmail.com} 

\begin{abstract}
In \cite{SU2} we described the structure of the graphs associated with the iterations of the map $x \mapsto x+x^{-1}$ over finite fields of characteristic two. In this paper we extend our study to finite fields of characteristic five. 
\end{abstract}

\maketitle
\section{Introduction}
Let $\F_q$ be a finite field with $q$ elements for some positive integer $q$. We can define a map $\theta$ on $\Pro (\F_q) = \F_q \cup \{\infty \}$ in such a way:
\begin{displaymath}
\theta(x) =   
\begin{cases}
x+x^{-1} & \text{if $x \not =0, \infty$}\\
\infty & \text{if $x = 0$ or $\infty$}
\end{cases}
\end{displaymath}

We associate a graph with the map $\theta$ over $\F_{q}$, labelling the vertices of the graph by the elements of $\Pro (\F_{q})$. Moreover, if $\alpha \in \Pro (\F_{q})$ and $\beta = \theta(\alpha)$, then a directed edge connects the vertex $\alpha$ with the vertex $\beta$.
If $\gamma \in \Pro (\F_q)$ and $\theta^k (\gamma) = \gamma$, for some positive integer $k$, then $\gamma$ belongs to a cycle of length $k$ or a divisor of $k$. The smallest among these integers $k$ is the period $l$ of $\gamma$ with respect to the map $\theta$ and the set $\{\theta^i (\gamma) : 0 \leq i < l \}$ is the cycle of length $l$ containing $\gamma$.
An element $\gamma$ belonging to a cycle can be the root of a reverse-directed tree, provided that $\gamma = \theta (\alpha)$, for some $\alpha$ which is not contained in any cycle. 

In \cite{SU2} we dealt with the characteristic $2$ case. There we noticed that the map $\theta$ is strictly related to the duplication map over Koblitz curves. Later we carried out some experiments in characteristics greater than $5$, but the resulting graphs seemed not to present notable symmetries.

In characteristics $3$ and $5$, in analogy with our previous work \cite{SU2}, the graphs exhibit remarkable symmetries.  In this paper we present the characteristic $5$ case.

\section{Structure of the graphs in characteristic five}
Consider the elliptic curve $E$ defined by the following equation over the field $\F_5$ with 5 elements:
\begin{equation}
y^2 = x^3+x.
\end{equation}
As showed in Chapter 25, Example 25.1.5 of \cite{Gal} the map 
\begin{equation}
\phi(x,y) = \left( \frac{x^2+1}{x}, y \frac{x^2-1}{x^2} \right)
\end{equation}
is an isogeny from $E$ to $E$ itself. 
Its dual isogeny  is 
\begin{equation}
\overline{\phi} (x,y) = \left( \frac{-x^2-1}{x}, 2y \frac{x^2-1}{x^2} \right),
\end{equation}

while the duplication map $[2]$ over $E$ is defined as
\begin{equation}
[2] (x,y) = \left(\frac{x^4+3x^2+1}{4(x^3+x)}, \frac{x^6-1}{3y (x^3+x)}  \right).
\end{equation}

Indeed the reader can check that $[2](x,y) = \phi(\overline{\phi} (x,y))$.

Let $\F_{5^n}$ be the field with $5^n$ elements for some positive integer $n$. 
We notice that, if $P = (x, y) \in E(\F_{5^n})$ is a rational point of the curve $E$ over $\F_{5^n}$, then $\theta(x)$ is the $x$-coordinate of the point $\phi(P)$. 

As a first step in our work we study the endomorphism ring $\E5(E)$ of the curve $E$. We note that $E$ is an ordinary elliptic curve (see \cite{was}). Therefore, by Theorem 2.4 of \cite{wi},  $\E5(E)$ is an order in an imaginary quadratic field. Among the endomorphisms of $\E5(E)$ there is the Frobenius endomorphism $\pi_5$, which takes a point $(x,y)$ to $(x^5,y^5)$. As a consequence of Theorem 2.4 in \cite{wi} the ring $\E5(E)$ is an order in $\Q (i)$ and the representation of the Frobenius endomorphism as an element of $\Q (i)$ is 
\begin{equation*}
\pi_5 = 1+2i.
\end{equation*}
The ring of integers of $\Q (i)$ is $\Z [i]$. The duplication map splits in $\Z [i]$ as 
\begin{equation*}
2 = (-1+i) (-1-i).
\end{equation*}
We have seen before that the duplication map $[2]$ is the composition of the maps $\phi$ and $\overline{\phi}$. Therefore, these two endomorphisms correspond to $(-1+i)$ and $(-1-i)$ in $\Z[i]$. Since $\E5(E)$ is contained in $\Z[i]$ and both the endomorphisms  $1+2i$ and $-1-i$ belongs to $\E5(E)$,  we conclude that $\E5(E) = \Z[i]$. From now on we will denote by $R$ the ring $\E5(E)$, namely $\Z[i]$.

Let $n$ be a fixed positive integer. 

If $n$ is odd we define the sets
\begin{eqnarray*}
A_n & = & \left\{x \in \F_{5^n}: (x,y) \in E(\F_{5^{n}}) \text{ for some $y \in \F_{5^{n}}$} \right\} \cup \{ \infty \} \\
B_n & = & \left\{x \in \F_{5^n}: (x,y) \in E(\F_{5^{2n}}) \text{ for some $y \in \F_{5^{2n}} \backslash \F_{5^n}$} \right\} \backslash \{1, -1 \}\\
C_n & = & \left\{1, -1 \right\}
\end{eqnarray*}

If, on the contrary, $n$ is even, then we define 
\begin{eqnarray*}
A_n & = & \left\{x \in \F_{5^n}: (x,y) \in E(\F_{5^{n}}) \text{ for some $y \in \F_{5^{n}}$} \right\} \cup \{ \infty \} \\
B_n & = & \left\{x \in \F_{5^n}: (x,y) \in E(\F_{5^{2n}}) \text{ for some $y \in \F_{5^{2n}} \backslash \F_{5^n}$} \right\} .
\end{eqnarray*}
We note in passing that, when $n$ is even,  $\{1, -1 \} \not \subseteq B_n$. 

The sets $A_n, B_n$ (and $C_n$, if $n$ is odd) form a partition of $\Pro(\F_{5^n})$. In fact, if $x \in \F_{5^n}$, then the equation $y^2 = x^3+x$ has two distinct non-zero solutions $\pm y$ in $\F_{5^{2n}}$ in the case that $x^3+x \not = 0$, while has just the solution $y=0$, if $x^3+x = 0$. Since $y \in \F_{5^n}$ if and only if $-y \in \F_{5^n}$, then $x \in A_n$ if and only if $x \not \in B_n$. In order to  prove this last claim, we have just to deal with the elements $\pm 1$ in the case $n$ is odd. Such elements belong to $C_n$ and, by definition, they do not belong to $B_n$. We can indeed show that $\pm 1 \not \in A_n$. In fact, the equations 
\begin{displaymath} 
y^2 = 2 \quad \text{and} \quad y^2 = 3
\end{displaymath} 
have solutions in $\F_{5^2} \backslash \F_5$. Since $\F_{5^2} \not \subseteq \F_{5^n}$ when $n$ is odd, we are done.

The following holds.
\begin{lemma}
The map $\theta$ acts separately on the sets $A_n$ and $B_n$, while sends elements of $C_n$ to $A_n$.
\end{lemma}
\begin{proof}
Firstly we notice that $\theta (\infty) = \infty \in A_n$. Consider now an element $x$ in $A_n$ different from $\infty$. Then $(x,y) \in E(\F_{5^n})$ for some $y \in \F_{5^n}$. Therefore, $\phi(x,y) \in E(\F_{5^n})$ and consequently $\theta(x) \in A_n$.

If $x \in B_n$, then $(x,y) \in E(\F_{5^{2n}})$ for some $y \in \F_{5^{2n}} \backslash \F_{5^n}$. Therefore, the point $\phi(x,y) = \left(\theta(x), y \cdot \dfrac{x^2-1}{x^2} \right) \in E(\F_{5^{2n}}) \backslash E(\F_{5^n})$. In fact, being $y \in \F_{5^{2n}} \backslash \F_{5^n}$, the element $y \cdot \dfrac{x^2-1}{x^2} \in \F_{5^n}$ if and only if $\dfrac{x^2-1}{x^2} = 0$, namely $x = 1$ or $-1$. Since $\pm 1 \not \in B_n$, we conclude that $\theta(x) \in B_n$. 

Finally, $\theta(1) =2$ while $\theta(-1) = 3$ in $\F_5$. Since $(3,0)$ and $(2,0)$ are points of $E(\F_{5^n})$, we conclude that $\theta(1)$ and $\theta(-1)$ belong to $A_n$. 
\end{proof}

Take now an element $x \in A_n$ different from $\infty$. By Theorem 2.3 in \cite{wi} there is an isomorphism 
\begin{displaymath}
\psi : E(\F_{5^n}) \to R / (\pi_5^n-1) R.
\end{displaymath}
  
Therefore, we can study the action of the map $\theta$ on $A_n$ relying upon the structure of the (additive) group $R/(\pi_5^n-1) R$.

On the contrary, if $x \in B_n$, then there is no $y \in \F_{5^n}$ such that $y^2 = x^3+x$, namely $x$ is the $x$-coordinate of a point in $E(\F_{5^{2n}})$. Let us define 
\begin{displaymath}
E(\F_{5^{2n}})_{B_n} = \left\{(x,y) \in E(\F_{5^{2n}}) :    \text{$x \in B_{n}$ and $y \in \F_{5^{2n}} \backslash \F_{5^n}$} \right\}.
\end{displaymath}
Moreover, let $\widetilde{\alpha}$ and $\widetilde{\beta}$ be the elements of $\F_{5^{2}}$ such that
\begin{eqnarray*}
(\pm \widetilde{\alpha})^2 & = & 1^3+1=2\\ 
(\pm \widetilde{\beta})^2 & = & (-1)^3-1 = -2.
\end{eqnarray*}

Define the set $E(\F_{5^{2n}})_{B_n}^*$ in such a way:
\begin{displaymath}
\begin{array}{lcll}
E(\F_{5^{2n}})_{B_n}^* & = & \{ O, (0,0), (2,0), (3,0), (1, \pm \widetilde{\alpha}), (-1, \pm \widetilde{\beta} \}, & \text{if $n$ is odd;}\\
E(\F_{5^{2n}})_{B_n}^* & = & \{ O, (0,0), (2,0), (3,0) \}, & \text{if $n$ is even.}
\end{array}
\end{displaymath}
In both cases $O$ denotes the point at infinity.

The following holds.
\begin{lemma}
Let $x \in \F_{5^{2n}}$ and $P=(x,y) \in E(\F_{5^{2n}})$.  Then, 
$\left(\pi_{5}^{n} +1 \right) P = O$ if and only if $P \in E(\F_{5^{2n}})_{B_n} \cup E(\F_{5^{2n}})_{B_n}^*$.
\end{lemma}
\begin{proof}
Consider a rational point $P=(x,y) \in E(\F_{5^{2n}})$ such that $\left( \pi_5^n +1 \right) (P) = O$. This is equivalent to saying that $\left( x^{5^n}, y^{5^n} \right) = (x,-y)$. Then, $x \in \F_{5^n}$ and $y^{5^n} = -y$. If $y \in \F_{5^n}$, then $y^{5^n} = y$ and consequently $y=-y$. Therefore $y=0$ and $x^3+x=0$, namely $x=0, 2$ or $-2$. In all cases $P \in E(\F_{5^{2n}})_{B_n}^*$. If, on the contrary, $y \not \in \F_{5^n}$, then $P \in E(\F_{5^{2n}})_{B_n}$ or, only if $n$ is odd, $x= \pm 1$, namely $P \in E(\F_{5^{2n}})_{B_n}^*$. 

Viceversa, suppose that $P \in E(\F_{5^{2n}})_{B_n}$. Then, $\pi_5^n (P) = \left(x, y^{5^n} \right)$. Since $y \not \in \F_{5^n}$, then $y^{5^n} \not = y$. This implies that $\pi_5^n (P) = - P$. Finally, it can be checked that $\left(\pi_{5}^{n} +1 \right) P = O$ for any point $P$ in $E(\F_{5^{2n}})_{B_n}^*$.
\end{proof}

As an immediate consequence of the result just proved we can claim that there is an isomorphism
\begin{displaymath}
\widetilde{\psi} : E(\F_{5^{2n}})_{B_n} \cup E(\F_{5^{2n}})_{B_n}^*  \to R / (\pi_5^n+1) R.
\end{displaymath}

\subsection{Cycles.}

We will study the structure of the cycles and trees of the graph associated with $\theta$ proceeding as in \cite{SU2}. The ring $R = \Z[i]$ is euclidean, hence a unique factorization domain. An element $a+ib \in R$ has norm $N(a+ib) = a^2+b^2$.  All positive primes of $\Z$, different from $2$, either split in $R$ or are inert. We can factor the element $\pi_5^n-1$ (respectively $\pi_5^n+1$) in primes of $R$. 

Let $\rho = -1+i$ and suppose that  $\pi_5^n-1$ (resp. $\pi_5^n+1$) factors as
\begin{equation*}
\rho^{e_0} \cdot \left( \prod_{i=1}^v p_i^{e_i} \right) \cdot \left( \prod_{i = v+1}^w r_i^{e_i} \right),
\end{equation*}
where 
\begin{enumerate}
\item each $e_i \in \N$;
\item  for $1 \leq i \leq v$ the elements $p_i \in \Z$ are distinct primes of $R$ and $N( p_i^{e_i} ) = p_i^{2 e_i}$;
\item for $v+1 \leq i \leq w$ the elements $r_i \in R \backslash \Z$  are distinct primes of $R$, different from $\rho$ and $\overline{\rho}$,  and $N( r_i^{e_i} ) = p_i^{e_i}$, for some rational integer $p_i$ such that $r_i \overline{r}_i = p_i$.
\end{enumerate}

The ring $R / (\pi_5^n - 1) R$ (resp. $R / (\pi_5^n + 1) R$) is isomorphic to 
\begin{equation}\label{orb4}
R / \rho^{e_0} R \times \left( \prod_{i=1}^v R / p_i^{e_i} R \right) \times \left( \prod_{i = v+1}^w R / r_i^{e_i} R \right).
\end{equation}

For any $1 \leq i \leq v$ the additive group of $R / p_i^{e_i} R$ is isomorphic to the direct sum of two cyclic groups of order $p_i^{e_i}$. This implies that, for each integer $0 \leq h_i \leq e_i$, there are $N_{h_i}$ points in $R / p_i^{e_i} R$ of order $p_i^{h_i}$, where
\begin{displaymath}
N_{h_i} = \left\{
\begin{array}{ll}
1 & \text{ if $h_i = 0$}\\
{p_i}^{2 h_i} - {p_i}^{2(h_{i}-1)} & \text{ otherwise.}
\end{array}
\right.
\end{displaymath}

For any $v+1 \leq i \leq w$ the additive group of $R / r_i^{e_i} R$ is cyclic of order $p_i^{e_i}$. Hence, there are $\phi(p_i^{h_i})$ points in $R / r_i^{e_i} R$ of order $p_i^{h_i}$, for each integer $0 \leq h_i \leq e_i$.

An element $x \in A_n \backslash \{0, \infty \}$ (resp. $B_n$), which is periodic under the action of the map $\theta$, is the $x$-coordinate of a rational point of $E (\F_{5^n})$ (resp. $E (\F_{5^{2n}}))$, which corresponds to a point of the form $P = (0, P_1, \dots, P_w) \in R / (\pi_5^n-1) R$ (resp. $R/ (\pi_5^n+1) R$). Each $P_i$, for $1 \leq i \leq w$, has order $p_i^{h_i}$, for some integer $0 \leq h_i \leq e_i$. 

For any $P_i$ we aim at finding the smallest positive integer $l_i$ such that $[\rho]^{l_i} P_i = P_i$ or $-P_i$. 
\begin{itemize}
\item If $1 \leq i \leq v$, then $l_i$ is the smallest among the positive integers $k$ such that $p_i^{h_i}$ divides  $\rho^{k} + 1$ or $\rho^{k} - 1$ in $R$.
\item If $v+1 \leq i \leq w$, then $l_i$ is the smallest among the positive integers $k$ such that ${r_i}^{h_i}$ divides  $\rho^{k} + 1$ or $\rho^{k} - 1$ in $R$.
\end{itemize}

Let
\begin{displaymath}
l' = lcm (l_1, \dots, l_w).
\end{displaymath}
We introduce parameters $\epsilon_i$, for $1 \leq i \leq w$, such defined: 
\begin{displaymath}
\epsilon_i = \left\{
\begin{array}{ll}
1 & \textrm{if $[\rho]^{l_i} P_i = P_i$}\\
0 & \textrm{if $[\rho]^{l_i} P_i = - P_i$}.
\end{array}
\right.
\end{displaymath}
Let 
\begin{displaymath}
\epsilon = \left\{
\begin{array}{ll}
0 & \textrm{if any $\epsilon_i =1$ or any $\epsilon_i = 0$}\\
1 & \textrm{otherwise}.
\end{array}
\right.
\end{displaymath}
Then, the period of $x$ with respect to $\theta$ is $l = 2^{\epsilon} \cdot l'.$

We note that the number of points $P = (0, P_1, \dots, P_w)$ in $R/(\pi_5^n -1) R$ (resp. $R/(\pi_5^n+1) R$), where each $P_i$ has order $p_i^{h_i}$,  is 
\begin{displaymath}
m = \left( \prod_{i=1}^{v} N_{h_i} \right)  \cdot \left( \prod_{i=v+1}^{w} \phi(p_i^{h_i})  \right). 
\end{displaymath} 
Let $P = \psi(x,y)$ (resp. $\widetilde{\psi} (x,y)$) be one of such points. The period $l$ of $x$ can be calculated as above. In particular, we note that also $-P = \psi(x,-y)$ (resp. $\widetilde{\psi} (x,-y)$) has the same additive order in $R / (\pi_5^n-1)R$ (resp. $R/(\pi_5^n+1) R$).  This amounts to saying that the $m$ points give rise to 
$\dfrac{m}{2l}$    
cycles of length $l$.

Now, define the sets $Z_{e_i} = \{0, 1, \dots, e_i \}$, for any $1 \leq i \leq w$. Let
\begin{equation*}
H = \prod_{i=1}^{w} Z_{e_i}.
\end{equation*}

For any $h \in H$ denote by $C_h$ the set of all cycles formed by the elements $x \in A_n$ (resp. $B_n$) such that $(x,y) \in E (\F_{5^n})$ (resp. $E (\F_{5^{2n}})_{B_n}$) for some $y \in \F_{5^n}$ (resp. $\F_{5^{2n}}$) and $\psi(x,y) = P = (0, P_1, \dots, P_w) \in R / (\pi_5^n-1) R$ (resp. $\widetilde{\psi} (x,y) = P \in R / (\pi_5^n+1) R$), where
\begin{itemize}
\item each $P_i$, for $1 \leq i \leq v$, has additive order $p_i^{h_i}$ in $R / p_i^{e_i} R$;
\item each $P_i$, for $v+1 \leq i \leq w$, has additive order $p_i^{h_i}$ in $R / r_i^{e_i} R$.
\end{itemize}  
Let $l_h$ be the length of the cycles formed by these points. 
Finally, denote by $C_{A_n}$ the set of all cycles formed by the elements of $A_n$ and by $C_{B_n}$ the set of all cycles formed by the elements of $B_n$.

The following holds.
\begin{lemma}
With the above notation, $C_{A_n}$ (resp. $C_{B_n}$) is equal to $\displaystyle\bigcup_{h \in H} C_h$, being 
\begin{displaymath}
\lvert C_{h} \rvert = \frac{1}{2 l_h} \left( \prod_{i=1}^{v} N_{h_i} \right)  \cdot \left( \prod_{i=v+1}^{w} \phi(p_i^{h_i}) \right)
\end{displaymath}
for any non-zero $h \in H$.
\end{lemma} 

In the following we will denote by $V_{A_n}$ (respectively $V_{B_n}$) the set of the elements of $\F_{5^n}$ belonging to some cycle of $C_{A_n}$ (respectively $C_{B_n}$).

We want to analyse the depth of the trees attached to the nodes of the cycles. To this purpose we recall that, if $x+iy$ is a complex number and $n$ is a positive integer, then the $n$-th power of $x+iy$ can be expressed as
\begin{equation*}
\sum_{k=0}^{\lfloor \frac{n}{2} \rfloor} (-1)^k {n \choose 2k} x^{n-2k} y^{2k} 
+   i \cdot \sum_{k=0}^{\lfloor \frac{n-1}{2} \rfloor} (-1)^k {n \choose 2k+1} x^{n-(2k+1)} y^{2k+1}.
\end{equation*}

The following Lemma, which is a special case of a Theorem by Legendre (see \cite{Le}), and the subsequent Corollary will be useful in determining the highest power of 2 which divides a factorial.
\begin{lemma}
Let $n=a_k \cdot 2^k + a_{k-1} \cdot 2^{k-1} + \dots + a_1 \cdot 2 + a_0$ be the base 2 expansion of a positive integer $n$. Then, the highest power of $2$ dividing $n!$ is $2^{\mu}$, where
\begin{equation*}
\mu = n - (a_k+a_{k-1}+ \dots + a_1 + a_0).
\end{equation*} 
\end{lemma}
\begin{corollary}\label{cor_hpn}
If $n$ is a positive integer and $2^{\mu}$ is the highest power of $2$ dividing $n!$, then $\mu \leq n-1$.
\end{corollary}

Before proceeding with the study of  the trees, we notice that any element in $\Z[i] / \rho^k \Z[i]$ can be uniquely represented in the form
\begin{displaymath}
[a_0+a_1 \rho + \dots + a_k \rho^{k-1}],
\end{displaymath}
where each $a_i \in \{0, 1 \}$ (see \cite{gil}).
 
\subsection{Trees rooted in elements  of $V_{A_n}$}
We begin this subsection proving the following technical result.
\begin{lemma}\label{normpi}
Let $n=2^l m$, for some non-negative integer $l$ and odd integer $m$, be a positive integer. Then, the highest power of $2$ which divides $N(\pi_5^n-1)$, where $\pi_5 = 1+2i$, is
\begin{itemize}
\item $2^2$, if $n$ is odd;
\item $2^{3+2l}$, if $n$ is even.
\end{itemize} 
\end{lemma}
\begin{proof}
The norm of $\pi_5^n-1$ is
\begin{equation*}
N(\pi_5^n-1) = \left[\sum_{k=1}^{\lfloor \frac{n}{2} \rfloor} (-1)^k {n \choose 2k}  2^{2k} \right]^2
+ \left[\sum_{k=0}^{\lfloor \frac{n-1}{2} \rfloor} (-1)^k {n \choose 2k+1} 2^{2k+1} \right]^2.
\end{equation*}
Suppose firstly that $n$ is odd. We notice that, by Corollary \ref{cor_hpn}, the highest power of $2$ dividing $(2k)!$ and $(2k+1)!$ does not exceed $2^{2k-1}$ and $2^{2k}$ respectively.  Since $n-1$ is even, the power of $2$ dividing ${n \choose 2k } \cdot 2^{2k}$, for $k$ greater than $0$, is not smaller than $2^{2}$ and the power of $2$ dividing  ${n \choose 2k+1} \cdot 2^{2k+1}$, for $k$ greater than $0$, is not smaller than $2^2$. Therefore, 
\begin{equation*}
N(\pi_5^n-1) = [2^2 \cdot r]^2 + [2m+2^2s]^2
\end{equation*}
for some integers $r$ and $s$.
Passing to the residue classes modulo 4 and 8 we get
\begin{eqnarray*}
N(\pi_5^n-1) & \equiv & 0 +  [2m]^2 \equiv 0 \pmod{4}\\
N(\pi_5^n-1) & \equiv & 0 +  [2m+4s]^2 \equiv 4 m^2 \pmod{8}.
\end{eqnarray*}
Therefore, the highest power of $2$ dividing $N(\pi_5^n-1)$ is $2^2$.

Suppose now that $n$ is even and that $n=2^l m$, for some positive integer $l$ and odd integer $m$. If $n=2$, then $N(\pi_5^2-1)=32=2^5$ and we are done.

Suppose that $n$ is greater than 2, namely $n \geq 4$. We notice that ${n \choose 2} \cdot 2^2 = n \cdot (n-1) \cdot 2$, hence $2^{l+1}$ is the highest power of 2 which divides ${n \choose 2} \cdot 2^2$. Moreover,  each term ${n \choose 2k} \cdot 2^{2k}$, for $k \geq 2$, is divisible by $2^{l+2}$ and each term ${n \choose 2k+1} \cdot 2^{2k+1}$, for $k \geq 1$, is divisible by $2^{l+2}$. Therefore,
\begin{equation*}
N(\pi_5^n-1) = \left[-2^{l+1} m (n-1) + 2^{l+2} r \right]^2
+ \left[2^{l+1} m + 2^{l+2} s \right]^2,
\end{equation*}  
for some  integers $r$ and $s$.

Passing to the residue classes modulo $2^{3+2l}$ and $2^{4+2l}$ we get respectively
\begin{eqnarray*}
N(\pi_5^n-1) & \equiv & 2^{2l+2} m^2 + 2^{2l+2} m^2 = m^2 \cdot 2^{2l+3} \equiv 0 \pmod{2^{3+2l}}\\
N(\pi_5^n-1) & \equiv & 2^{2l+2} m^2 + 2^{2l+2} m^2 = m^2 \cdot 2^{2l+3}  \not\equiv 0 \pmod{2^{4+2l}},
\end{eqnarray*} 
since $m^2$ is an odd integer.
\end{proof}

The following Lemma characterizes the reversed trees having root in $V_{A_n}$.
\begin{lemma}\label{orb_2}
Any element $x \in V_{A_n}$ is the root of a reversed binary tree of depth $e_0$ with the following properties.
\begin{itemize}
\item If $x \not = \infty$,  then there are $2^{k-1}$ vertices at any level $k>0$ of the tree.  Moreover, the root has one child, while all other vertices have two children.
\item If $x = \infty$, then there is $1$ vertex at the level $1$, there are $2$ vertices at the level $2$ of the tree and $2^{k-2}$ vertices at any level $k > 2$. Moreover, the root and the vertices at the level $2$ have one child, while all other vertices have two children. 
\item If $n$ is odd, then $e_0=2$, if $x \not = \infty$, and $e_0=3$, if $x = \infty$.
\item If $n$ is even and $2^l$ is the greatest power of $2$ which divides $n$, then $e_0 = 3+2l$.
\end{itemize}
\end{lemma}
\begin{proof}
For a fixed element $x \in V_{A_n}$ let $(0, P_1, \dots, P_w) \in R / (\pi_5^{n}-1) R$ be one of the (at most two) points with such an $x$-coordinate (here we use the isomorphic representation of $R / (\pi_5^{n}-1) R$ given by (\ref{orb4})). An element $\tilde{x} \in \F_{5^n}$ belongs to the non-zero level $k$ of the reversed binary tree rooted in $x$ if and only if $\theta^k(\tilde{x}) = x$, $\theta^i (\tilde{x}) \not = x$ and none of the $\theta^i (\tilde{x})$ is periodic for any $i < k$. Since $\theta$ maps elements of $A_n$ to $A_n$ and elements of $B_n$ to $B_n$, then $\tilde{x} \in A_n$ (at the end of the proof we will deal with the elements belonging to $C_n$). Therefore there exists $\tilde{y} \in \F_{5^n}$ such that $(\tilde{x}, \tilde{y}) \in E (\F_{5^n})$ and $\psi(\tilde{x}, \tilde{y}) = Q = (Q_0, Q_1, \dots, Q_w)$, where $Q_0 \not =0$. Since $[\rho]^{e_0} Q_0 = 0$ in $R / \rho^{e_0} R$ we have that $k \leq e_0$.

For a fixed positive integer $k \leq e_0$ we aim to find all the points $(Q_0, Q_1,  \dots, Q_w)$ in $R /(\pi_5^{n}-1) R $ such that 
\begin{enumerate}
\item $[\rho]^k Q_0  =  0$ and $[\rho]^{k-1} Q_0 \not = 0$ ;
\item $[\rho]^k Q_i  =   P_i$ for any $1 \leq i \leq w$, or $[\rho]^k Q_i  =   -P_i $ for any $1 \leq i \leq w$. 
\end{enumerate}
The first condition is satisfied if and only if
\begin{equation}\label{q0}
Q_0  = [\rho]^{e_0-k} + \sum_{i=e_0-k+1}^{e_0-1} j_i [\rho]^i, 
\end{equation}
where each $j_i \in \left\{ 0, 1 \right\}$.
The second condition is satisfied if and only if 
\begin{equation*}
Q_i  =  [\rho]^{-k} P_i \text{, for any $i$,}  
\end{equation*}
or
\begin{equation*}
Q_i  =  -[\rho]^{-k} P_i \text{, for any $i$.} 
\end{equation*}

Hence, fixed the values of $j_i$ for $e_0-k+1 \leq i \leq e_0-1$, there are at most two possibilities for $Q$, namely
\begin{eqnarray*}
Q^{(1)} & = & (Q_0, [\rho]^{-k} P_1, \dots, [\rho]^{-k} P_w) \quad \textrm{or}  \\
Q^{(2)} & = & (Q_0, -[\rho]^{-k} P_1, \dots, -[\rho]^{-k} P_w) .
\end{eqnarray*}
Therefore, for a fixed positive integer $k$ there are $2^k$ points $Q$, whose $x$-coordinate belongs to the level $k$ of the tree, provided that not all $P_i$ are zero. If $Q = \psi(\tilde{x}, \tilde{y})$ is one of such points, then $-Q=\psi(\tilde{x}, -\tilde{y})$ has the same $x$-coordinate. Hence, for any $x \in C_{A_n}$ different from $\infty$ there are $2^{k-1}$ vertices at the level $k$ of the reversed binary tree rooted in $x$.

If all $P_i$ are zero, then $x = \infty$. In particular, the points $Q^{(1)}$ and $Q^{(2)}$ above coincide. Therefore, for any $k>0$ there are $2^{k-1}$ points, whose $x$-coordinate belongs to the level $k$ of the tree. Moreover, if $Q$ is one of such points, also $-Q$ has the same $x$-coordinate and is different from $Q$, unless $Q=(2,0)$ or $(3,0)$, namely $Q$ is one of the two points belonging to the level $2$ of the tree. This amounts to say that there are $\lceil 2^{k-2} \rceil$ vertices at the level $k$ of the tree, if $k\not=2$, and $2$ vertices at the level $2$.

Consider now an element $\tilde{x}$ belonging to the  level $k < e_0$ of the tree rooted in some $x \in V_{A_n}$. Such an $\tilde{x}$ is the $x$-coordinate of a point $Q=(Q_0, Q_1, \dots, Q_w)$ in $R/(\pi_5^n-1) R$, for some $Q_0$ as in (\ref{q0}) or $Q_0 = 0$. The equation $z+z^{-1} = \tilde{x}$ is satisfied for at most two $z$ in $\F_{5^n}$, which are the $x$-coordinate of two points in $R/(\pi_5^n-1)R$,
\begin{equation*}
\widetilde{Q}^{(1)} = (\widetilde{Q}_0^{(1)}, \widetilde{Q}_1, \dots, \widetilde{Q}_w) \quad \widetilde{Q}^{(2)} = (\widetilde{Q}_0^{(2)}, \widetilde{Q}_1, \dots, \widetilde{Q}_w),
\end{equation*}
where 
\begin{eqnarray*}
\widetilde{Q}_0^{(1)} & = & [\rho]^{e_0-k-1} + \sum_{i=e_0-k}^{e_0-2} j_{i+1} [\rho]^i \\
\widetilde{Q}_0^{(2)} & = & [\rho]^{e_0-k-1} + \sum_{i=e_0-k}^{e_0-2} j_{i+1} [\rho]^i + [\rho]^{e_0-1}
\end{eqnarray*}
and
\begin{equation*}
\widetilde{Q}_i  =  [\rho]^{-1} Q_i, \quad \text {if $1 \leq i \leq w$}.
\end{equation*}

We note that $[\rho] \widetilde{Q}^{(1)} = [\rho] \widetilde{Q}^{(2)} = Q$.

If $k=0$, then just the point $\widetilde{Q}^{(1)}$ belongs to the tree, proving that the root of the tree has one and only one child.

If $k \geq 1$ and at least one of the $P_i$ is non-zero, then $\widetilde{Q}^{(1)} \not = - \widetilde{Q}^{(2)}$, hence $\widetilde{Q}^{(1)}$ and $\widetilde{Q}^{(2)}$ have different $x$-coordinates. This implies that each vertex at non-zero level $k$ of a tree rooted in $x \in V_{A_n} \backslash \{ \infty \}$ has two children.

Suppose now that all the $P_i$ are zero. In this case $x= \infty$.
If $k = 1$, then $\widetilde{Q}_0^{(1)} = - \widetilde{Q}_0^{(1)}$ and also $\widetilde{Q}_0^{(2)} = - \widetilde{Q}_0^{(2)}$. In fact, since
\begin{eqnarray*}
\widetilde{Q}_0^{(1)} & = & [\rho]^{e_0-2}\\
\widetilde{Q}_0^{(2)} & = & [\rho]^{e_0-2}+[\rho]^{e_0-1},
\end{eqnarray*}
then
\begin{eqnarray*}
\widetilde{Q}_0^{(1)} + \widetilde{Q}_0^{(1)} & = & [2] [\rho]^{e_0-2} = [\overline{\rho}] [\rho]^{e_0-1} = [i] [\rho]^{e_0} = 0\\
\widetilde{Q}_0^{(2)} + \widetilde{Q}_0^{(2)} & = & [2] [\rho]^{e_0-2}+[2][\rho]^{e_0-1} = 0.
\end{eqnarray*}
Hence, the two points at the level $2$ have not the same $x$ coordinate, implying the the only vertex at the level $1$ has two children.

Consider now the case of $k=2$. Then,
\begin{eqnarray*}
\widetilde{Q}_0^{(1)} & = & [\rho]^{e_0-3} +  j_{e_0-1} [\rho]^{e_0-2}\\
\widetilde{Q}_0^{(2)} & = & [\rho]^{e_0-3} +  j_{e_0-1} [\rho]^{e_0-2} + [\rho]^{e_0-1}.
\end{eqnarray*}
We want to prove that $\widetilde{Q}_0^{(1)} = - \widetilde{Q}_0^{(2)}$. In fact,
\begin{eqnarray*}
\widetilde{Q}_0^{(1)} + \widetilde{Q}_0^{(2)} & = & [\overline{\rho}][\rho] ([\rho]^{e_0-3} +  j_{e_0-1} [\rho]^{e_0-2})  + [\rho]^{e_0-1} \\
& = & [\overline{\rho}]([\rho]^{e_0-2} +  j_{e_0-1} [\rho]^{e_0-1}) + [\rho]^{e_0-1}\\
& = & [\rho+\rho^2] ([\rho]^{e_0-2} +  j_{e_0-1} [\rho]^{e_0-1}) + [\rho]^{e_0-1}\\
& = &  [\rho]^{e_0-1} +  [\rho]^{e_0-1} = 0.
\end{eqnarray*}
Therefore, each of the vertices at the level $2$ has exactly $1$ child.

Finally, if $k > 2$ and  all the $P_i$ are zero, then $\widetilde{Q}^{(1)} \not = - \widetilde{Q}^{(2)}$. In fact,
\begin{eqnarray*}
\widetilde{Q}_0^{(1)} + \widetilde{Q}_0^{(2)} & = & [\overline{\rho}][\rho] \left([\rho]^{e_0-k-1} + \sum_{i=e_0-k}^{e_0-2} j_{i+1} [\rho]^i \right)  + [\rho]^{e_0-1} \\
& = & [\overline{\rho}] \left([\rho]^{e_0-k}  + \sum_{i=e_0-k+1}^{e_0-1} j_{i} [\rho]^i \right)  + [\rho]^{e_0-1}\\
& = & [\rho+\rho^2] \left([\rho]^{e_0-k}  + \sum_{i=e_0-k+1}^{e_0-1} j_{i} [\rho]^i \right)  + [\rho]^{e_0-1}\\
& = &  [\rho]^{e_0-k+1} + [\rho]^{e_0-k+2} + j_{e_0-k+1} [\rho]^{e_0-k+2} \\
& & + \sum_{i=e_0-k+3}^{e_0-1} +  (j_{i-1}+j_{i-2}) [\rho]^{i} + [\rho]^{e_0-1} \not = 0,
\end{eqnarray*}
since $k \geq 3$.
Hence, each of the vertices at the levels $k>2$ of the tree rooted in $\infty$ have two children.

As regards the number $e_0$, we remind that the greatest power of $2$ which divides $N(\pi_5^n-1)$ is $2^2$ if $n$ is odd and $2^{3+2l}$ in the other cases. Moreover, $R/(\pi_5^n-1) R$ is isomorphic to a product of rings as in (\ref{orb4}). 
Since $N(\rho) = 2$, then ${e_0}$ is the greatest power of 2 which divides $N(\pi_5^n-1)$.

We know that $\theta$ maps points of $A_n$ to $A_n$ and points of $B_n$ to $B_n$. We just have to deal with the points of $C_n = \{1, -1 \}$, when $n$ is odd. For these two points we have that $\theta(1)=2$ and $\theta(-1)=-2 \equiv 3 \pmod{5}$, namely $1$ and $-1$ belongs to the tree rooted in $\infty$. We note that $x+x^{-1} = 1$ if and only if $x^2-x+1=0$ and that $x+x^{-1} = -1$ if and only if $x^2+x+1=0$. The equations 
$x^2 \pm x+1 =  0$
have a solution $x$ in the field $\F_{5^n}$ if and only if $n$ is even. Therefore, there exists $x$  in $\F_{5^n}$ such that $x+x^{-1}=  1$ (or $-1$) if and only if $n$ is even. Therefore, when $n$ is odd, the tree rooted in $\infty$ has depth $3$ and is as follows:
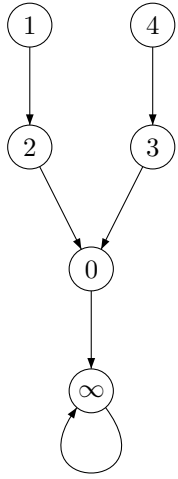
\begin{center}
\begin{picture}(30, 60)(-20,-10)
	\unitlength=4.6pt
    \gasset{Nw=3.6,Nh=3.6,Nmr=1.8,curvedepth=0}
    \thinlines
    \footnotesize
    \node(N1)(0,0){$\infty$}
    \node(N2)(0,10){0}
    \node(N3)(-5,20){2}
    \node(N4)(-5,30){1}
    \node(N5)(5,20){3}
    \node(N6)(5,30){4}
    \drawedge(N6,N5){}
    \drawedge(N5,N2){}
    \drawedge(N2,N1){}
    \drawedge(N4,N3){}
    \drawedge(N3,N2){}
    \drawloop[loopangle=-90](N1){}
\end{picture}
\end{center}
 
\end{proof}

\subsection{Trees rooted in elements of $V_{B_n}$}
The following Lemma characterizes the reversed trees having root in $V_{B_n}$.
\begin{lemma}\label{orb_3}
Any element $x \in V_{B_n}$ is the root of a reversed binary tree with the following properties.
\begin{itemize}
\item There are $2^{k-1}$ vertices at any level $k > 0$ of the tree.  Moreover, the root has one child, while all other vertices have two children.
\item The depth of the tree is $2$ if $n$ is even and $3$ if $n$ is odd.
\end{itemize}
\end{lemma}
\begin{proof}
The proof is very similar to the proof of Lemma \ref{orb_2}.

For a fixed element $x \in V_{B_n}$, let $(0, P_1, \dots, P_w) \in R / (\pi_5^{n}+1) R$ be one of the (at most two) points with such an $x$-coordinate (here we use the isomorphic representation of $R / (\pi_5^{n}+1) R$ given by (\ref{orb4})). We notice that not all $P_i$ can be zero, since that would imply that $x \in V_{A_n}$.  An element $\tilde{x} \in \F_{5^n}$ belongs to the non-zero level $k$ of the reversed binary tree rooted in $x$ if and only if $\theta^k(\tilde{x}) = x$, $\theta^i (\tilde{x}) \not = x$ and none of the $\theta^i (\tilde{x})$ is periodic for any $i < k$. Since $\theta$ maps elements of $A_n$ to $A_n$ and elements of $B_n$ to $B_n$, then $\tilde{x} \in B_n$. Therefore there exists $\tilde{y} \in \F_{5^{2n}} \backslash \F_{5^{n}}$ such that $(\tilde{x}, \tilde{y}) \in E (\F_{5^{2n}})_{B_n}$ and $\widetilde{\psi}(\tilde{x}, \tilde{y}) = Q = (Q_0, Q_1, \dots, Q_w)$, where $Q_0 \not =0$. Since $[\rho]^{e_0} Q_0 = 0$ in $R / \rho^{e_0} R$ we have that $k \leq e_0$.

For a fixed positive integer $k \leq e_0$ we aim to find all the points $(Q_0, Q_1,  \dots, Q_w)$ in $R /(\pi_5^{n}-1) R $ such that 
\begin{enumerate}
\item $[\rho]^k Q_0  =  0$ and $[\rho]^{k-1} Q_0 \not = 0$ ;
\item $[\rho]^k Q_i  =   P_i$ for any $1 \leq i \leq w$, or $[\rho]^k Q_i  =   -P_i $ for any $1 \leq i \leq w$. 
\end{enumerate}
The first condition is satisfied if and only if
\begin{equation}\label{q1}
Q_0  = [\rho]^{e_0-k} + \sum_{i=e_0-k+1}^{e_0-1} j_i [\rho]^i, 
\end{equation}
where each $j_i \in \left\{ 0, 1 \right\}$.
The second condition is satisfied if and only if 
\begin{equation*}
Q_i  =  [\rho]^{-k} P_i \text{, for any $i$,}  
\end{equation*}
or
\begin{equation*}
Q_i  =  -[\rho]^{-k} P_i \text{, for any $i$.} 
\end{equation*} 

Since not all $P_i$ are zero, once the values of $j_i$ for $e_0-k+1 \leq i \leq e_0-1$ have been fixed, there are  two possibilities for $Q$, namely
\begin{eqnarray*}
Q^{(1)} & = & (Q_0, [\rho]^{-k} P_1, \dots, [\rho]^{-k} P_w) \quad \textrm{or}  \\
Q^{(2)} & = & (Q_0, -[\rho]^{-k} P_1, \dots, -[\rho]^{-k} P_w) .
\end{eqnarray*}
Therefore, for a fixed positive integer $k$ there are $2^k$ points $Q$, whose $x$-coordinate belongs to the level $k$ of the tree. If $Q = \widetilde{\psi}(\tilde{x}, \tilde{y})$ is one of such points, then $-Q=\widetilde{\psi}(\tilde{x}, -\tilde{y})$ has the same $x$-coordinate. Moreover, $Q \not = -Q$, because $\tilde{y} \not = 0$, being $\tilde{y} \in \F_{5^{2n}} \backslash \F_{5^n}$. Hence, for any $x \in V_{B_n}$ there are $2^{k-1}$ vertices at the level $k$ of the reversed binary tree rooted in $x$.

Suppose now that $n$ is odd. The greatest power of $2$ which divides $N(\pi_5^n+1)$ is $2^3$. In fact,
\begin{equation}\label{q2}
N(\pi_5^n+1) N(\pi_5^n-1) = N(\pi_5^{2n}-1),
\end{equation}
the greatest power of $2$ dividing $N(\pi_5^n-1)$ is $2^2$, while the greatest power of $2$ dividing the right hand side of the equation is $2^5$. These facts imply that $2^3$ is the greatest power of $2$ which divides $N(\pi_5^n+1)$.

Suppose on the contrary that $n$ is even. Then there exists some positive integer $l$ such that
\begin{displaymath}
\begin{array}{llll}
N(\pi_5^n-1) & = & 2^{3+2l} \cdot r & \text{for some odd integer $r$}\\
N(\pi_5^{2n}-1) & = & 2^{3+2(l+1)} \cdot s & \text{for some odd integer $s$}.
\end{array}
\end{displaymath} 
We deduce that $2^2$ is the greatest power of $2$ dividing $N(\pi_5^n+1)$.

Since the greatest power of $2$ dividing $N(\pi_5^n+1)$ is equal to $e_0$, the depth of the tree is $2$ if $n$ is even and $3$ if $n$ is odd.
\end{proof}

\subsection{An example: the graph associated with $\theta$ over the field $\F_{5^3}$} 
The field with $125$ elements can be constructed as the splitting field over $\F_5$ of the Conway polynomial $x^3-2x-2$. In particular, if $\alpha$ denotes a root of such a polynomial,
\begin{equation*}
\Pro(\F_{5^3}) = \left\{ \alpha^i: 0 \leq i \leq 123 \right\}  \cup \{ 0 \}  \cup \left\{ \infty \right\}.
\end{equation*}
In the following pages  the $4$ connected components of the graph are represented. The labels of the nodes are the exponents of the powers $\alpha^i$, for $0 \leq i \leq 123$, the zero element (denoted by `0') and the point $\infty$. It can be noticed that the first, second and fourth component correspond to the elements belonging to $A_3$. In particular we notice that, in accordance with Lemma \ref{orb_2}, being $n=3$ odd, each cycle is the root of a binary tree of depth $2$, except for the tree rooted in $\infty$, which has depth $3$. The third connected component corresponds to the elements belonging to $B_3$ and each node in the cycle is the root of a binary tree of depth $3$ as claimed in Lemma \ref{orb_3}.

As regards the length of the cycles of $C_{A_3}$, we notice that $\pi_5^3-1 = -12-2i$, which decomposes in primes of $\Z[i]$ as $\rho^2 \cdot (1-6i)$. Therefore,
\begin{displaymath}
R / (\pi_5^3-1) R \cong R / \rho^2 R \times R / (1-6 i) R. 
\end{displaymath} 
Consider a point $(0,P) \not = (0,0)$ belonging to the isomorphic representation of $R / (\pi_5^3-1) R$. The element $P$ has additive order $37$ in $R/(1-6i) R$. The length of the cycle containing $(0,P)$ is $l$, where $l$ is the smallest among the positive integers $k$ such that either $\rho^k+1$ or $\rho^k-1$ is divisible by $(1-6i)$ in $R=\Z[i]$. The integer $l$ is equal to $9$. Since there are $36$ points as $(0,P)$, 
then there are $\dfrac{36}{2 \cdot 9} = 2$ cycles of length $9$.

Consider now the cycles of $C_{B_3}$. We notice that $\pi_5^3+1 = -10 - 2 i$, which decomposes in prime of $\Z[i]$ as $\rho^3 \cdot (-3+2i)$. Therefore,
\begin{displaymath}
R / (\pi_5^3+1) R \cong R / \rho^3 R \times R / (-3+2i) R.
\end{displaymath} 
Consider a point $(0,P) \not = (0,0)$ belonging to the isomorphic representation of $R / (\pi_5^3+1) R$. The element $P$ has additive order $13$ in $R/(-3+2i) R$. The length of the cycle containing $(0,P)$ is $l = 6$, since $l$ is the smallest among the positive integers $k$ such that either $\rho^k+1$ or $\rho^k-1$ is divisible by $(-3+2i)$ in $R=\Z[i]$.  The $12$ points with the same characteristics of $(0,P)$ form $\dfrac{12}{2 \cdot 6} = 1$ cycle of length $6$.

\newpage
\begin{example}
\begin{center}
    \unitlength=4.6pt
    \begin{picture}(70, 70)(-35,-35)
    \gasset{Nw=3.6,Nh=3.6,Nmr=1.8,curvedepth=-0.5}
    \thinlines
    \footnotesize
    \node(N1)(10,0){$78$}
    \node(N2)(7.66,6.42){$72$}
    \node(N3)(1.73,9.84){$110$}
    \node(N4)(-5,8.66){$90$}
    \node(N5)(-9.39,3.42){$64$}
    \node(N6)(-9.39,-3.42){$22$}
    \node(N7)(-5,-8.66){$18$}
    \node(N8)(1.73,-9.84){$112$}
    \node(N9)(7.66,-6.42){$54$}
    
    \node(N11)(20,0){$70$}
    \node(N12)(15.32,12.84){$46$}
    \node(N13)(3.46,19.68){$52$}
    \node(N14)(-10,17.32){$14$}
    \node(N15)(-18.78,6.84){$34$}
    \node(N16)(-18.78,-6.84){$60$}
    \node(N17)(-10,-17.32){$102$}
    \node(N18)(3.46,-19.68){$106$}
    \node(N19)(15.32,-12.84){$12$}

 	\node(N101)(29.54,5.2){$1$}
    \node(N102)(25.98,15.0){$107$}
    \node(N103)(19.28,22.98){$17$}
    \node(N104)(10.26,28.19){$49$}
    \node(N105)(0,30){$75$}
    \node(N106)(-10.26,28.19){$25$}
    \node(N107)(-19.28,22.98){$99$}
    \node(N108)(-25.98,15){$53$}
    \node(N109)(-29.54,5.2){$71$}
    \node(N111)(-29.54,-5.2){$15$}
    \node(N112)(-25.98,-15){$109$}
    \node(N113)(-19.28,-22.98){$5$}
    \node(N114)(-10.26,-28.19){$119$}
    \node(N115)(0,-30){$39$}
    \node(N116)(10.26,-28.19){$85$}
    \node(N117)(19.28,-22.98){$3$}
    \node(N118)(25.98,-15){$121$}
    \node(N119)(29.54,-5.2){$123$}

    \drawedge(N1,N2){}
    \drawedge(N2,N3){}
    \drawedge(N3,N4){}
    \drawedge(N4,N5){}
    \drawedge(N5,N6){}
    \drawedge(N6,N7){}
    \drawedge(N7,N8){}
    \drawedge(N8,N9){}
    \drawedge(N9,N1){}
    
    \gasset{curvedepth=0}
     
    \drawedge(N11,N1){}
    \drawedge(N12,N2){}
    \drawedge(N13,N3){}
    \drawedge(N14,N4){}
    \drawedge(N15,N5){}
    \drawedge(N16,N6){}
    \drawedge(N17,N7){}
    \drawedge(N18,N8){}
    \drawedge(N19,N9){}
    
    \drawedge(N101,N11){}
    \drawedge(N102,N12){}
    \drawedge(N103,N12){}
    \drawedge(N104,N13){}
    \drawedge(N105,N13){}
    \drawedge(N106,N14){}
    \drawedge(N107,N14){}
    \drawedge(N108,N15){}
    \drawedge(N109,N15){}
    \drawedge(N111,N16){}
    \drawedge(N112,N16){}
    \drawedge(N113,N17){}
    \drawedge(N114,N17){}
    \drawedge(N115,N18){}
    \drawedge(N116,N18){}
    \drawedge(N117,N19){}
    \drawedge(N118,N19){}
    \drawedge(N119,N11){}
\end{picture}
\begin{picture}(70, 70)(-35,-35)
    \gasset{Nw=3.6,Nh=3.6,Nmr=1.8,curvedepth=-0.5}
    \thinlines
    \footnotesize
    \node(N1)(10,0){$2$}
    \node(N2)(7.66,6.42){$84$}
    \node(N3)(1.73,9.84){$80$}
    \node(N4)(-5,8.66){$50$}
    \node(N5)(-9.39,3.42){$116$}
    \node(N6)(-9.39,-3.42){$16$}
    \node(N7)(-5,-8.66){$10$}
    \node(N8)(1.73,-9.84){$48$}
    \node(N9)(7.66,-6.42){$28$}
    
    \node(N11)(20,0){$96$}
    \node(N12)(15.32,12.84){$122$}
    \node(N13)(3.46,19.68){$40$}
    \node(N14)(-10,17.32){$44$}
    \node(N15)(-18.78,6.84){$74$}
    \node(N16)(-18.78,-6.84){$8$}
    \node(N17)(-10,-17.32){$108$}
    \node(N18)(3.46,-19.68){$114$}
    \node(N19)(15.32,-12.84){$76$}

 	\node(N101)(29.54,5.2){$9$}
    \node(N102)(25.98,15.0){$77$}
    \node(N103)(19.28,22.98){$47$}
    \node(N104)(10.26,28.19){$67$}
    \node(N105)(0,30){$57$}
    \node(N106)(-10.26,28.19){$23$}
    \node(N107)(-19.28,22.98){$101$}
    \node(N108)(-25.98,15){$59$}
    \node(N109)(-29.54,5.2){$65$}
    \node(N111)(-29.54,-5.2){$61$}
    \node(N112)(-25.98,-15){$63$}
    \node(N113)(-19.28,-22.98){$45$}
    \node(N114)(-10.26,-28.19){$79$}
    \node(N115)(0,-30){$13$}
    \node(N116)(10.26,-28.19){$111$}
    \node(N117)(19.28,-22.98){$37$}
    \node(N118)(25.98,-15){$87$}
    \node(N119)(29.54,-5.2){$115$}

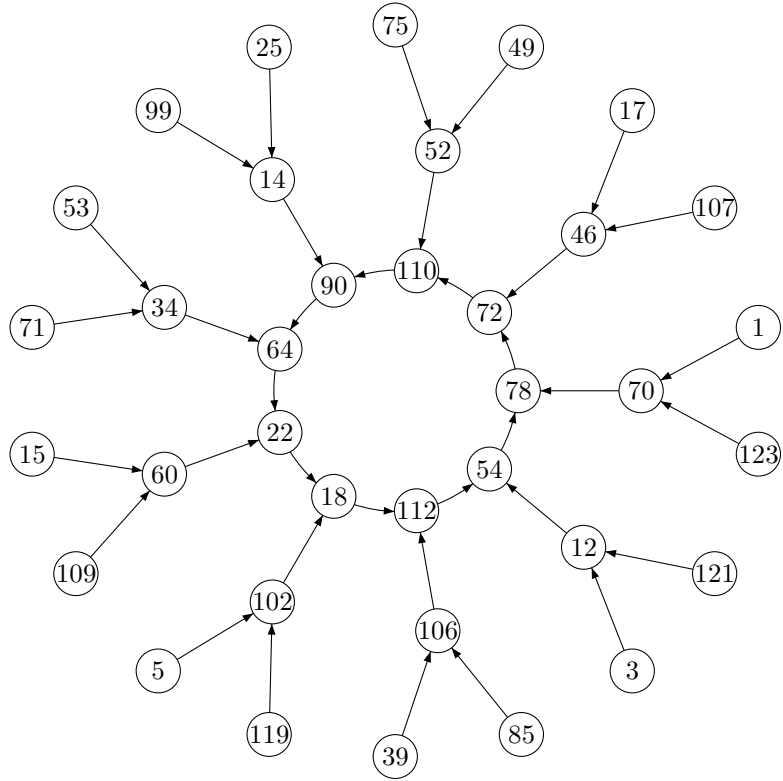
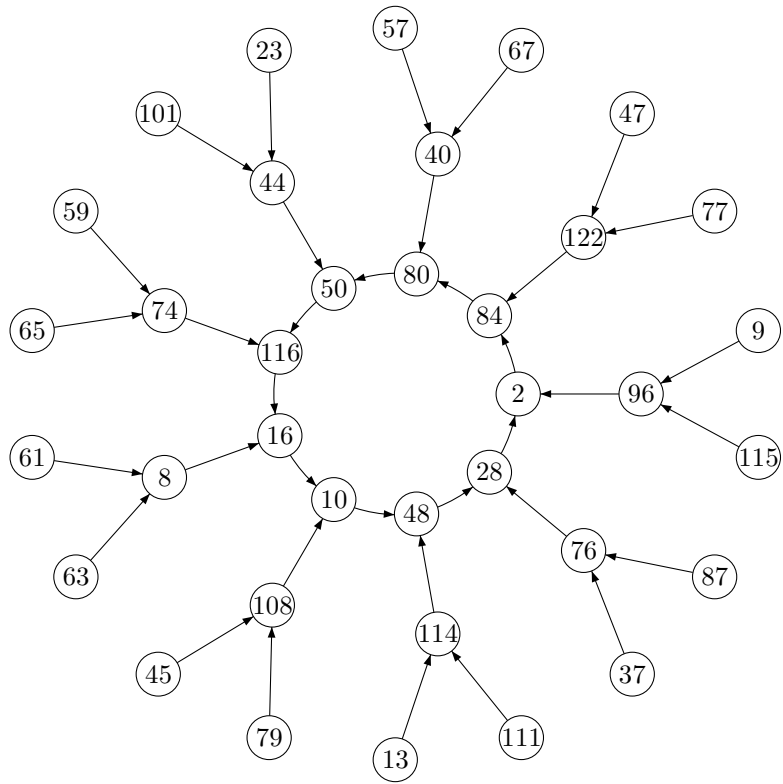
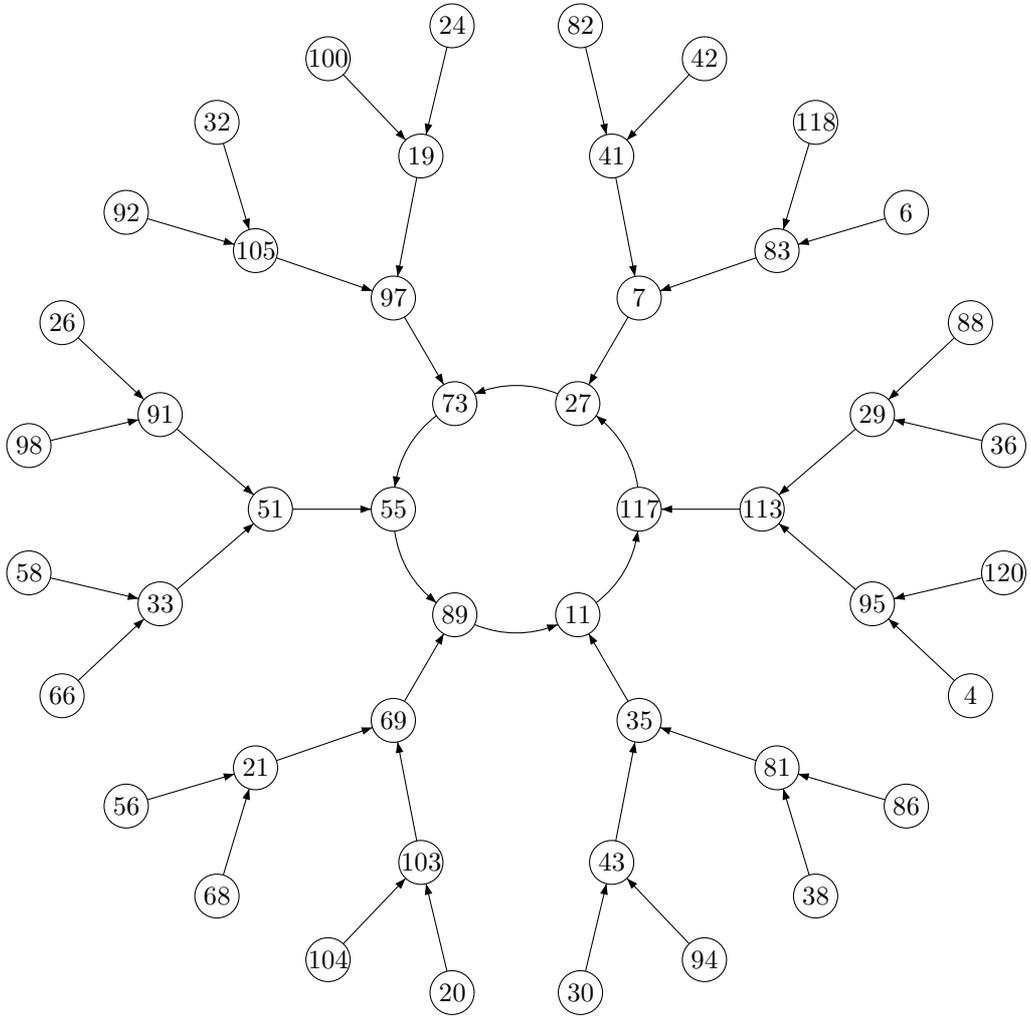
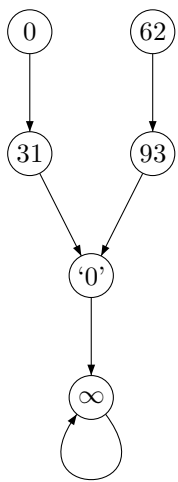
    \drawedge(N1,N2){}
    \drawedge(N2,N3){}
    \drawedge(N3,N4){}
    \drawedge(N4,N5){}
    \drawedge(N5,N6){}
    \drawedge(N6,N7){}
    \drawedge(N7,N8){}
    \drawedge(N8,N9){}
    \drawedge(N9,N1){}
    
    \gasset{curvedepth=0}
     
    \drawedge(N11,N1){}
    \drawedge(N12,N2){}
    \drawedge(N13,N3){}
    \drawedge(N14,N4){}
    \drawedge(N15,N5){}
    \drawedge(N16,N6){}
    \drawedge(N17,N7){}
    \drawedge(N18,N8){}
    \drawedge(N19,N9){}
    
    \drawedge(N101,N11){}
    \drawedge(N102,N12){}
    \drawedge(N103,N12){}
    \drawedge(N104,N13){}
    \drawedge(N105,N13){}
    \drawedge(N106,N14){}
    \drawedge(N107,N14){}
    \drawedge(N108,N15){}
    \drawedge(N109,N15){}
    \drawedge(N111,N16){}
    \drawedge(N112,N16){}
    \drawedge(N113,N17){}
    \drawedge(N114,N17){}
    \drawedge(N115,N18){}
    \drawedge(N116,N18){}
    \drawedge(N117,N19){}
    \drawedge(N118,N19){}
    \drawedge(N119,N11){}
\end{picture}
\end{center}
\begin{center}
\begin{picture}(120, 160)(-60,-78)
	\unitlength=4.6pt
    \gasset{Nw=3.6,Nh=3.6,Nmr=1.8,curvedepth=-1.5}
    \thinlines
    \footnotesize
    \node(N1)(10,0){$117$}
    \node(N2)(5,8.66){$27$}
    \node(N3)(-5,8.66){$73$}
    \node(N4)(-10,0){$55$}
    \node(N5)(-5,-8.66){$89$}
    \node(N6)(5,-8.66){$11$}
    
    \node(N11)(20,0){$113$}
    \node(N12)(10,17.32){$7$}
    \node(N13)(-10,17.32){$97$}
    \node(N14)(-20,0){$51$}
    \node(N15)(-10,-17.32){$69$}
    \node(N16)(10,-17.32){$35$}

 	\node(N21)(28.97,7.76){$29$}
    \node(N22)(21.21,21.21){$83$}
    \node(N23)(7.76,28.97){$41$}
    \node(N24)(-7.76,28.97){$19$}
    \node(N25)(-21.21,21.21){$105$}
    \node(N26)(-28.97,7.76){$91$}
    \node(N27)(-28.97,-7.76){$33$}
    \node(N28)(-21.21,-21.21){$21$}
    \node(N29)(-7.76,-28.97){$103$}
    \node(N30)(7.76,-28.97){$43$}
    \node(N31)(21.21,-21.21){$81$}
    \node(N32)(28.97,-7.76){$95$}

    \node(N41)(39.65,5.22){$36$}
    \node(N42)(36.95,15.30){$88$}
    \node(N43)(31.73,24.35){$6$}
    \node(N44)(24.35,31.73){$118$}
    \node(N45)(15.31,36.95){$42$}
    \node(N46)(5.22,39.65){$82$}
    \node(N47)(-5.22,39.65){$24$}
    \node(N48)(-15.31,36.95){$100$}
    \node(N49)(-24.35,31.73){$32$}
    \node(N50)(-31.73,24.35){$92$}
    \node(N51)(-36.95,15.30){$26$}
    \node(N52)(-39.65,5.22){$98$}
    \node(N53)(-39.65,-5.22){$58$}
    \node(N54)(-36.95,-15.3){$66$}
    \node(N55)(-31.73,-24.35){$56$}
    \node(N56)(-24.35,-31.73){$68$}
    \node(N57)(-15.31,-36.95){$104$}
    \node(N58)(-5.22,-39.65){$20$}
    \node(N59)(5.22,-39.65){$30$}
    \node(N60)(15.31,-36.95){$94$}
    \node(N61)(24.35,-31.73){$38$}
    \node(N62)(31.73,-24.35){$86$}
    \node(N63)(36.95,-15.3){$4$}
    \node(N64)(39.65,-5.22){$120$}
    
    \drawedge(N1,N2){}
    \drawedge(N2,N3){}
    \drawedge(N3,N4){}
    \drawedge(N4,N5){}
    \drawedge(N5,N6){}
    \drawedge(N6,N1){}

    \gasset{curvedepth=0}
     
    \drawedge(N11,N1){}
    \drawedge(N12,N2){}
    \drawedge(N13,N3){}
    \drawedge(N14,N4){}
    \drawedge(N15,N5){}
    \drawedge(N16,N6){}
    
    \drawedge(N21,N11){}
    \drawedge(N22,N12){}
    \drawedge(N23,N12){}
    \drawedge(N24,N13){}
    \drawedge(N25,N13){}
    \drawedge(N26,N14){}
    \drawedge(N27,N14){}
    \drawedge(N28,N15){}
    \drawedge(N29,N15){}
    \drawedge(N30,N16){}
    \drawedge(N31,N16){}
    \drawedge(N32,N11){}
    
    \drawedge(N41,N21){}
    \drawedge(N42,N21){}
    \drawedge(N43,N22){}
    \drawedge(N44,N22){}
    \drawedge(N45,N23){}
    \drawedge(N46,N23){}
    \drawedge(N47,N24){}
    \drawedge(N48,N24){}
    \drawedge(N49,N25){}
    \drawedge(N50,N25){}
    \drawedge(N51,N26){}
    \drawedge(N52,N26){}
    \drawedge(N53,N27){}
    \drawedge(N54,N27){}
    \drawedge(N55,N28){}
    \drawedge(N56,N28){}
    \drawedge(N57,N29){}
    \drawedge(N58,N29){}
    \drawedge(N59,N30){}
    \drawedge(N60,N30){}
    \drawedge(N61,N31){}
    \drawedge(N62,N31){}
    \drawedge(N63,N32){}
    \drawedge(N64,N32){}
%    

%    \drawedge(N107,N14){}
%    \drawedge(N108,N15){}
%    \drawedge(N109,N15){}
%    \drawedge(N111,N16){}
%    \drawedge(N112,N16){}
%    \drawedge(N113,N17){}
%    \drawedge(N114,N17){}
%    \drawedge(N115,N18){}
%    \drawedge(N116,N18){}
%    \drawedge(N117,N19){}
%    \drawedge(N118,N19){}
%    \drawedge(N119,N11){}
\end{picture}
\begin{picture}(30, 60)(-20,-0)
	\unitlength=4.6pt
    \gasset{Nw=3.6,Nh=3.6,Nmr=1.8,curvedepth=0}
    \thinlines
    \footnotesize
    \node(N1)(0,0){$\infty$}
    \node(N2)(0,10){`0'}
    \node(N3)(-5,20){$31$}
    \node(N4)(-5,30){$0$}
    \node(N5)(5,20){$93$}
    \node(N6)(5,30){$62$}
    \drawedge(N6,N5){}
    \drawedge(N5,N2){}
    \drawedge(N2,N1){}
    \drawedge(N4,N3){}
    \drawedge(N3,N2){}
    \drawloop[loopangle=-90](N1){}
\end{picture}
\end{center}
\end{example}  
\section{A note about the construction of the graphs}
The Examples of this paper have been manually constructed for the sake of clearness. 
In this Section we describe a possible procedure for obtaining those and other examples using GAP \cite{gap} and Graphviz \cite{gvz}.  Details about GAP, Graphviz and DOT language can be found in the official websites of GAP and Graphviz.

The following GAP function generates a \emph{.dot} file encoding all information about the directed graph associated with the map $\theta$ over the field $\F_{p^n}$ for chosen positive integers $p$ and $n$. In particular, the code can be used for constructing graphs in any characteristic $p$.
\lstset{frame=single, breaklines=true, tabsize=2,basicstyle=\ttfamily}
\begin{lstlisting}
graph := function(p,n)
local k, gen,a ,b, j, f;
f:="./graph.dot";
gen:=Z(p^n); 
PrintTo(f, "digraph{ \n");
for j in [0..p^n-2] do
	a:=gen^j;
	b:=((a)+(a)^(-1)); 
	if b = 0*Z(p) then
		 AppendTo(f, j, "-> zero; \n");
	else
		for k in [0..p^n] do
			if gen^k = b then break;fi;
		od;
	AppendTo(f, j, "->", k, "; \n");
	fi;
od;
AppendTo(f, "zero -> inf; \n inf -> inf;}");
end;;
\end{lstlisting}

It can be convenient to filter all connected components of the graph using the tool \verb+ccomps+ included in Graphviz. Once the file \emph{graph.dot} has been created, the data files of the connected components can be generated as follows:
\begin{verbatim}
ccomps -x -o graph_con graph.dot
\end{verbatim}
The $i$-th connected component's file will be named \verb+graph_con_i+. 

Finally, it is possible to give the connected components a layout using one of the Graphviz layout commands as \verb+dot+, \verb+circo+, \verb+neato+ or  \verb+twopi+.
\bibliography{Refs}
\end{document}